\documentclass{amsart}[10pt]
\usepackage{amsmath}
\usepackage{amsmath,amsthm}
\usepackage{amssymb}
\usepackage{amsfonts}
\usepackage{hyperref}
\usepackage{mathtools}

\newtheorem{theorem}{Theorem}[section]
\newtheorem{lemma}[theorem]{Lemma}

\theoremstyle{definition}
\newtheorem{definition}[theorem]{Definition}
\newtheorem{example}[theorem]{Example}

\newtheorem{proposition}[theorem]{Proposition}
\newtheorem{corollary}[theorem]{Corollary}
\theoremstyle{remark}
\newtheorem{remark}[theorem]{Remark}

\begin{document}
\title[$r$-Hyperideals{ in Krasner Hyperrings}]{{$\MakeLowercase{r}$-hyperideals and generalizations of $\MakeLowercase{r}$-hyperideals in Krasner hyperrings}}
\author[P. Xu]{Peng Xu}
\address{Institute of Computational Science and Technology, Guangzhou University, Guangzhou,Guangdong,510006, P.R. China.,, School of Computer Science of Information Technology, QiannanNormal University for Nationalities, Duyun, Guizhou, 558000, P.R. China}
\email{gdxupeng@gzhu.edu.cn}
\author[M. Bolat]{Melis Bolat}
\address{Department of Mathematics, Yildiz Technical University, Istanbul, Turkey.
Orcid: 0000-0002-0667-524X}
\email{melisbolat1@gmail.com }
\author[E. Kaya]{Elif Kaya}
\address{Department of Mathematics and Science Education, Istanbul Sabahattin Zaim
University, Istanbul, Turkey. Orcid: 0000-0002-8733-2734}
\email{elif.kaya@izu.edu.tr}
\author[S. Onar]{Serkan Onar}
\address{Department of Mathematical Engineering, Yildiz Technical University, Istanbul,
Turkey. Orcid: 0000-0003-3084-7694}
\email{serkan10ar@gmail.com}
\author[B.A. Ersoy]{Bayram Ali Ersoy}
\address{Department of Mathematics, Yildiz Technical University, Istanbul, Turkey.
Orcid: 0000-0002-8307-9644}
\email{ersoya@yildiz.edu.tr}
\author[K. Hila]{Kostaq Hila}
\address{Department of Mathematical Engineering, Polytechnic University of Tirana,
Tirana, Albania. Orcid: 0000-0001-6425-2619}
\email{kostaq\_hila@yahoo.com}

\begin{abstract}
In this study, we examine some properties of $r$-hyperideals in the commutative Krasner hyperrings. Some properties of
$pr$-hyperideals are also studied. The relation between prime hyperideals and
$r$-hyperideals is investigated. We show that the image and the inverse image of an $r$-hyperideal is also
an $r$-hyperideal. We also introduce a generalization of $r$-hyperideals
and we prove some properties of them. 
\end{abstract}
\maketitle

\section{Introduction}

Prime ideals and primary ideals play a significant role in commutative ring
theory. Numerous generalizations of prime ideals have been studied by different
researchers. The concepts of $\phi$-prime, $\phi$-primary ideals are examined
in \cite{X2,X10}.

French mathematician Marty \cite{X22} initated the study of hyperstructures, as an
expansion of classical algebraic structures, in special, of hypergroups, at the 8th
Congress of Scandinavian Mathematicians in 1934. Subsequently, numerous
articles and various books have been published about this issue. This theory
has been studied by Corsini \cite{X9}, Vougiouklis \cite{X31}, Mittas
\cite{X25,X24}, Stratigopoulos \cite{X30}, Davvaz \cite{X14,X12,X23}, Krasner
\cite{X20}, Ameri \cite{X1}, De Salvo \cite{X15}, Barghi \cite{X8} and
Asokkumar and Velrajan \cite{X5,X6}, etc.

Hyperrings $(\Re,+,\cdot)$ are of different types. For example, when the
addition + is a hyperoperation and the multiplication $\cdot$ is a binary
operation, we name them Krasner (additive) hyperrings \cite{X20}. If +
is a binary operation and the multiplication $\cdot$ is a hyperoperation, the
hyperring is called multiplicative hyperring \cite{X29}.

Mohamadian investigated the $r$-ideals in \cite{X26}. In \cite{X3}, Ugurlu
studied generalizations of $r$-ideals. The notion of $r$-hyperideals was briefly mentioned in \cite{X28}, as a generalization of $r$-ideals in commutative rings. In this paper, our aim is to extend the
notion of $r$-ideals to $r$-hyperideals in Krasner hyperrings and generalize them. Let $\Re$ be a commutative Krasner hyperring, $N$ be a proper
hyperideal of $\Re$ and $\phi:Id(\Re)\rightarrow Id(\Re)\cup\{\emptyset\}$ be
a function. $Id(\Re)$ denotes the hyperideals of $\Re$. $N$ is called $\phi
-r$-hyperideal (resp., $\phi-pr$-hyperideal) if $a\cdot b\in N-\phi(N)$ with
$ann(a)=0$ implies that $b\in N$ (resp., $b^{n}\in N$), for $a,b\in\Re$. Some properties of them are provided. As a result, we obtain that the union of directed
collection of ascending chain $\phi-r$-hyperideals of $R$ is also $\phi
-r$-hyperideal of $\Re$, when $\phi$ preserves the order.

\section{Preliminaries}

First of all, we will give the definitions of $\phi-$prime ideal, $\phi-$primary ideal and $r$-ideal.

\begin{definition}
\cite{X18} Let $\Re$ be a commutative ring. A reduction of ideals is a function $\phi$ that leads any ideal $N$
of $\Re$ to other ideal $\phi(N)$ of $\Re$ such that the following statements hold:

$i)$ For all ideals $N$ of $\Re,$ $\phi(N)\subseteq N$,

$ii)$ If $N\subseteq M$ where $N$ and $M$ are ideals of $\Re$, then
$\phi(N)\subseteq\phi(M).$
\end{definition}

\begin{definition}
Let $\Re$ be a commutative ring. Let $N$ be an ideal of $\Re$. For $a,b\in\Re$:

$i)$ $N$ is said to be a $\phi-$prime ideal, if $ab\in N-\phi(N)$, then $a\in
N$ or $b\in N$ \cite{X2},

$ii)$ $N$ is said to be a $\phi-$primary ideal, if $ab\in N-\phi(N)$, then
$a\in N$ or $b\in\sqrt{N}$ \cite{X10}$.$
\end{definition}

\begin{definition}
\cite{X26} Let $\Re$ be a commutative ring. A proper ideal $N$ of $\Re$ is called an $r$-ideal (resp., $pr$-ideal),
if $ab\in N$ with $ann(a)=(0)$ implies that $b\in N$ (resp., $b^{n}\in N,$ for
any $n\in%
\mathbb{N}
$), for each $a,b\in\Re.$
\end{definition}

\begin{definition}
\cite{X3} Let $\Re$ be a commutative ring. A proper ideal $N$ of $\Re$ is called a $\phi-r$-ideal, if $ab\in
N-\phi(N)$ with $ann(a)=(0)$ implies that $b\in N$\ for $a,b\in\Re.$
Furthermore, $N$ is called a $\phi-pr$-ideal, if $ab\in N-\phi(N)$ with
$ann(a)=(0)$ implies that $b^{n}\in N$ $(n\in%
\mathbb{N}
)$, for $a,b\in\Re.$
\end{definition}

In the following, we recall some notions regarding hyperstructures.

Let $G$ be a nonempty set and $\circ: G\times G\rightarrow \mathcal{P}^{*}(G)$ be a mapping, where $\mathcal{P}^{*}(G)$ denotes the set of all nonempty subsets of $G$. Then, $\circ$ is called a \textit{binary (algebraic) hyperoperation} on $G$. An algebraic system $(G, \circ)$, where $\circ$ is a binary hyperoperation defined on $G$, is called a hypergroupoid \cite{X22}. Let
$(G,\circ)$ be a hypergroupoid and $\emptyset\neq$ $A,B$ be subsets of $G.$

$A\circ B=\bigcup\limits_{\substack{a\in A\\b\in B}}a\circ b$ and for $x\in
G$, $x\circ A=\left\{  x\right\}  \circ A,$ $A\circ x=A\circ\left\{
x\right\}.$ For each $x,y,z\in G,$ if $x\circ(y\circ z)=(x\circ y)\circ z,$
which means, $\bigcup\limits_{u\in x\circ y}u\circ z=\bigcup\limits_{v\in y\circ
z}x\circ v$, then $G$ is called a semihypergroup. When $(G,\circ)$ is a
hypergroupoid, if there is $e\in G$ such that $x\in(e\circ x)\cap(x\circ e)$, for every
$x\in G,$ then $e$ is called identity element.

Semihypergroup $(G,\circ)$ is said to be a hypergroup if $x\circ G=G\circ
x=G,$ $\forall x\in G$ \cite{X22}. If $(G,\circ)$ is a hypergroup,
$\emptyset\neq H$ is a subset of $G\ $and $x\circ H=H\circ x=H,$ $\forall
x\in H,$ then $(H,\circ)$ is called a subhypergroup of $(G,\circ)$.

Let $(G,\circ)$ be a hypergroup. If\ $x\circ y=y\circ x,$ $\forall x,y\in
G$, then $(G,\circ)$ is called a commutative hypergroup.

\begin{definition}
A non-empty set $\Re$ along with the hyperoperation $+$ is called a canonical
hypergroup if the following axioms hold:

i) $ x+(y+z)=(x+y)+z,$ for $x,y,z\in\Re;$

ii) $ x+y=y+x,$ for $x,y\in\Re;$

iii) there exists $0\in\Re$ such that $x+0=\{x\},$ for any $x\in\Re;$

iv) for any $x\in\Re,$ there exists a unique element $x^{\prime}\in\Re,$
such that $0\in x+x^{\prime}$ ($x^{\prime}$ is called the inverse of $x$)

v) $z\in x+y$ implies that $y\in-x+z$ and $x\in z-y,$ that is, $(\Re,+)$ is
reversible \cite{X25}.
\end{definition}

\begin{definition}
$(\Re,+,\cdot)$ is called a Krasner hyperring if it satisfies the following conditions:

1) $(\Re,+)$ is a canonical hypergroup;

2) $(\Re, \cdot)$ is a semigroup having $0$ as a bilaterally absorbing element,
i.e., $x\cdot 0=0\cdot x=0,$ for all $x\in\Re;$

3) $(y+z)\cdot x=(y\cdot x)+(z\cdot x)$ and $x\cdot(y+z)=(x\cdot y)+(x\cdot z),$ for all $x,y,z\in\Re$ \cite{X20}.
\end{definition}

\begin{lemma}
\cite{X13} A non-empty subset $A$ of a Krasner hyperring $\Re$ is a left
(resp. right) hyperideal if and only if

i) $x-y\subseteq A$ $(x-y\subseteq A),$ for $x,y\in A$,

ii) $r\cdot x\in A$ $(resp.$ $x\cdot r\in A),$ for $x\in A,r\in\Re.$
\end{lemma}

\begin{definition}
\cite{X19} Let $N$ be a hyperideal of a hyperring $\Re.$ Then $\sqrt{N}%
=\{r\in\Re:r^{n}\in N,$ \text{for some} $n\in%
\mathbb{N}
\}.$
\end{definition}

Let $\Re$ be a hyperring. For $a\in\Re,$ we define $ann(a)=\{r\in\Re:ra=0\}$.
If $ann(a)=0$ (resp., $ann(a)\neq0$)$,$ $a$ is said to be regular (resp., zero
divisor). We use the notion $r(\Re)$ (resp., $zd(\Re)$ to denote the set of all regular
elements (resp., zero divisors) \cite{X17}.

\begin{definition}
\cite{X28} A proper hyperideal $N$ of a commutative Krasner hyperring $\Re$ is
called an r-hyperideal (pr-hyperideal), if $a\cdot b\in N$ and $ann(a)=0$
implies that $b\in N$ ($b^{n}\in N$, for some $n\in%
\mathbb{N}
$), for any $a,b\in\Re$.
\end{definition}

\begin{example}
Clearly, $%
\mathbb{Z}
_{6}$ is a Krasner hyperring with the usual addition and multiplication and
$\{\hat{0},\hat{2},\hat{4}\}=N$ is a proper hyperideal.

$\qquad ann(\hat{5})=0$ and $\hat{0}.\hat{5}=\hat{0}\in$
$N$ implies that $\hat{0}\in N.$

$\qquad ann(\hat{5})=0$ and $\hat{2}.\hat{5}=\hat{4}\in$
$N$ implies that $\hat{2}\in N$ .

\qquad$ann(\hat{5})=0$ and $\hat{4}.\hat{5}=\hat{2}\in$
$N$ implies that $\hat{4}\in N$ .

Then $N$ is an $r$-hyperideal of $%
\mathbb{Z}
_{6}.$
\end{example}

\begin{example}\cite{Hila} Let $H=\{0, 1, a\}$ and $B=\{0, a\}$. Then, $(H, +, \cdot)$ is a Krasner hyperring with the hyperaddition and multiplication defined by 
\begin{equation*}
\begin{tabular}{l|lll}
$+ $ & 0 & 1 & $a$ \\ \hline
0 & 0 & 1 & $a$ \\
1 & 1 & $H$ & 1\\
$a$ & $a$ & 1 & $B$
\end{tabular}
\ \ \ \
\begin{tabular}{l|lll}
$\cdot $ & 0 & 1 & $a$ \\ \hline
0 & 0 & 0 & 0 \\
1 & 0 & 1 & $a$\\
$a$ & 0 & $a$ & 0
\end{tabular}
\end{equation*}%
Clearly, $B$ is an $r$-hyperideal of $H$.
\end{example}

\begin{example}\cite{kemprasit} Let $H=[0, 1]$. Then, $(H, \oplus, \cdot)$ is a Krasner hyperring where $\cdot$ is the usual multiplication and the hyperaddition $\oplus$ is defined by 
\begin{equation*}
x\oplus y=\begin{cases}
\{\max\{x,y\}\}, & \text{if}\ x\neq y\\
[0,x], & \text{if}\ x=y
\end{cases}
\end{equation*}%
Clearly, $I=[0, a]$, where $H\ni a<1$, is an $r$-hyperideal of $H$.
\end{example}

$Min(N)$ is the set of all minimal prime hyperideals which contain $N$. $Min((0))$ is
stated by $Min(\Re).$

\begin{definition}\cite{X21}
A ring $\Re$ satisfies:
\begin{itemize}
	\item Property A: if any finitely generated ideal
$N\subseteq zd(\Re)$ has non-zero annihilator; 
\item Annihilator condition: if
for any finitely generated ideal $N$ of $\Re$, there exists an element $a\in
\Re$ such that $ann(N)=ann(a)$; 
\item Strong annihilator condition (briefly
s.a.c.): if for any finitely generated ideal $N$ of $\Re$, there exists an
element $a\in N$ such that $ann(N)=ann(a)$ . 
\end{itemize}
\end{definition}

We use the same definition
for a hyperideal $N$ of hyperring $\Re.$

$N$ is a $z^{0}$-hyperideal if $a\in N,$ $b\in\Re$ and $ann(a)=ann(b)$ imply
that $b\in N.$

$soc(\Re)$ denotes the sum of all minimal hyperideals of $\Re$. The socle of a
reduced hyperring $\Re$ is called $soc(\Re)=\oplus_{i\in A}e_{i}\Re,$ where
$\{e_{i}:i\in A\}$ is the set of idempotents of $\Re$ \cite{X11}.

Let $A$ be a subset of $\Re$ and $a\in A$ such that $a\in a\Re a.$ Then $a$ is
called von Neumann regular element$.$ Therefore, if all of the elements are
von Neumann regular, $A$ or $\Re$ is called von Neumann regular \cite{X7}.

A two-sided hyperideal $N$ in a semihypergroup $\Re$ is called right pure
hyperideal if for each $a\in N$, there exists $b\in N$ such that $a\in ab$
\cite{X16}. Similarly, a hyperideal is called pure hyperideal in a Krasner
hyperring, if for each $a\in N$, there exists $b\in N$ such that $a=ab.$

A hyperring $\Re$ is called a reduced hyperring if there is no nilpotent
elements in $\Re$. If $x^{n}=0$ for $x\in\Re$,$\ n\in%
\mathbb{N}
$, then $x=0$ \cite{X4}.

Let $f:\Re\rightarrow S$ be a homomorphism, $N$ be a hyperideal of $\Re$ and
$M$ be a hyperideal of $S.$ Then, the hyperideal $\langle f(N)\rangle$ is said to be the
extension of $N$ and it is denoted by $N^{e}.$ The hyperideal $f^{-1}(M)$ is called
contraction of $M$ and denoted by $M^{c}$ \cite{X27} . The mapping
$\varphi:\Re\rightarrow S^{-1}\Re$ given by $a=a/1$ is a homomorphism. If $N$
is a hyperideal of $\Re$, then $\varphi(N)=S^{-1}N=\{\frac{i}{s}\in S^{-1}%
\Re:i\in N,s\in S\}$ is also a hyperideal of $S^{-1}\Re$ \cite{X14} . If
$S=r(\Re),$ then the hyperring $S^{-1}\Re$ is called quotient hyperring which
is denoted by $Q(\Re).$

\section{$r$-Hyperideals}

In this section, we examine some properties of $r$-hyperideals in commutative
hyperrings. We compare $r$-hyperideals with prime and maximal hyperideals.
Throughout the section, $(\Re,+,\cdot)$ is a commutative Krasner hyperring.

Our main result regarding $r$-hyperideals is the following:

\begin{theorem}
Let $\Re$ be a hyperring and $N$ be a hyperideal of $\Re$. Then the following statements
are equivalent:

a) $N$ is an $r$-hyperideal.

b) $(a\cdot\Re)\cap N=a\cdot N,$ for any $a\in r(\Re).$

c) $N=(N:a),$ for any $a\in r(\Re)\backslash N.$

d) $N=M^{c};$ where $M$ is a hyperideal of $Q(\Re).$
\end{theorem}

\begin{proof}
($a\Rightarrow b)$ Let $N$ be an $r$-hyperideal and $a$ be a regular element.
Suppose that $x\in$ $(a\cdot\Re)\cap N.$ Then $x\in a\cdot\Re$ and $x\in N.$
Thus, $x=a\cdot a^{^{\prime}},$ for $a^{\prime}\in\Re.$ Since $x=a\cdot
a^{^{\prime}}\in N$, $ann(a)=0$ and $N$ is an $r$-hyperideal, $a^{\prime}\in N.$
Hence, $a\cdot a^{\prime}\in a\cdot N$ and $(a\cdot\Re)\cap N\subseteq a\cdot
N.$

Therefore, for every $a\in\Re,$ $a\cdot N\subseteq(a\cdot\Re)\cap N,$
$(a\cdot\Re)\cap N=a\cdot N.$

$(b\Rightarrow c)$ We know that $N\subseteq(N:a),\ $for every $a\in\Re.$ Let
$a$ be regular, $a\notin N$ and $x\in(N:a).$ Hence $ann(a)=0$ and $x\cdot a\in
N.$ From (b), $a\cdot x\in(a\cdot\Re)\cap N$ and $a\cdot x\in a\cdot N.$ This
implies that $a\cdot x=a\cdot y,$ for $y\in N.$ Since $ann(a)=0,$ then $x=y\in
N.$ Hence, $x\in N.$ Thus, $(N:a)\subseteq N.$

$(c\Rightarrow d)$ Let $S$ be the set of regular elements and $\varphi
:\Re\rightarrow Q(\Re)$ be natural homomorphism. We know that $N\subseteq
\varphi^{-1}(M)=M^{c}$, for $M$ a hyperideal of $Q(\Re)$. Suppose $x\in$ $\varphi^{-1}(M).$ Since $\varphi
\left(  x\right)  =\frac{x}{s}\in S^{-1}\Re,$ then $s\cdot x\in N,$ for $s\in S.$
From (c), $x\in(N:s)=N.$

$(d\Rightarrow a)$ Let $a\cdot x\in N$ and $ann(a)=0.$ We have $(a\cdot
x)/1=a/1\cdot x/1\in M$ and since $a$ is regular, then there exists $1/a$ which is the
inverse of $a/1$ in $S.$ Thus $a/1\cdot x/1\cdot1/a\in M$ and $x/1\in M.$
Hence, $x\in N$ and so $N$ is an $r$-hyperideal.
\end{proof}

\begin{corollary} The following statements hold:

$a)$ The zero hyperideal is an $r$-hyperideal.

$b)$ The intersection of $r$-hyperideals is an $r$-hyperideal.

$c)$ When $N$ is an $r$-hyperideal, $N\subseteq zd(\Re).$

$d)$ Every $r$-hyperideal is a $pr$-hyperideal.

$e)$ A prime hyperideal is an $r$-hyperideal if and only if it consists all of
zerodivisors. As a result, every minimal prime hyperideal is an $r$-hyperideal.

$f)$ Let $N$ is an $r$-hyperideal, $S\subseteq\Re$ and $S\nsubseteq N$. Then
$(N:S)$ is an $r$-hyperideal. Particularly, $ann(x)$ and $ann(S)$ are always $r$-hyperideals.

$g)$ Every minimal hyperideal of a reduced hyperring is an $r$-hyperideal.

$h)$ Every pure hyperideal and every von Neumann regular hyperideal are $r$-hyperideals.

$i)$ Suppose that $\Re$ satisfies the s.a.c. and $N$ is a hyperideal of $\Re.$
$N$ is an $r$-hyperideal if and only if for every hyperideal $J$ and $K$ of
$\Re$ such that $J$ is finitely generated, whenever $J\cdot K\subseteq N$ and
$ann(J)=0,$ then $K\subseteq N.$

$j)$ The sum of two $r$-hyperideals may not be an $r$-hyperideal.
\end{corollary}

\begin{remark}
Let $K$ and $L$ be hyperideals of $Q(\Re).$ We know that $K^{c}\cdot
L^{c}\subseteq(K\cdot L)^{c}$ and $K^{c}+L^{c}\subseteq(K+L)^{c}.$ If $N$ and
$M$ are $r$-hyperideals of $\Re,$ then by Theorem 3.1(d), $N=K^{c}$ and $M=L^{c}$,
for some hyperideals $K$ and $L$ in $Q(\Re).$ It follows that:

$a)$ $N\cdot M$ is an $r$-hyperideal of $\Re$ if and only if $(K\cdot
L)^{c}\subseteq$ $K^{c}\cdot L^{c}.$(essentially $(K\cdot L)^{c}=K^{c}\cdot
L^{c}$)

$b)$ $N+M$ is an $r$-hyperideal of $\Re$ if and only if $(K+L)^{c}\subseteq$
$K^{c}+L^{c}$ (essentially $(K+L)^{c}=K^{c}+L^{c}$).
\end{remark}

\begin{lemma}
Let $\Re$ be a hyperring and $N$ be a hyperideal. Then the following statements hold:

$a)$ $N$ is an $r$-hyperideal if and only if for any $J$, $K$ hyperideals of
$\Re$ with $J\cap r(\Re)\neq\emptyset$ and $J\cdot K\subseteq N$, then
$K\subseteq N.$

$b)$ Assume that $N\subseteq zd(\Re)$ is not an $r$-hyperideal. There exist
hyperideals $J$ and $K$ such that $J\cap r(\Re)\neq\emptyset,$ $N\subsetneqq
J,$ $K$ and $J\cdot K\subseteq N.$
\end{lemma}

\begin{proof}
$a)$ Let $N$ be an $r$-hyperideal, $J$ and $K$ be hyperideals of $\Re$ such that
$J\cap r(\Re)\neq\emptyset$ and $J\cdot K\subseteq N.$ Let $a\in J\cap r(\Re)$ and $b\in K$. Since $J$ and $K$ are
hyperideals, then we can take $a\cdot\Re\subseteq J$ and $b\cdot\Re\subseteq K.$ By
our assumption, $a\cdot b\cdot\Re\subseteq N.$ Since $J\cap r(\Re
)\neq\emptyset,$ let we take $ann(a)=0.$ Then $N$ is an $r$-hyperideal and so, $b\cdot
\Re\subseteq N.$ Thus, $K\subseteq N.$

Conversely, assume that $J\cap r(\Re)\neq\emptyset$ and $J\cdot
K\subseteq N, K\subseteq N.$ Let $a\cdot\Re\subseteq J$ and $b\cdot\Re\subseteq
K.$ Then $a\cdot b\cdot\Re\subseteq N.$ Let $a\in J\cap r(\Re).$ So
$ann(a)=0.$ At the same time, since $K\subseteq N$, then $b\cdot\Re\subseteq N$.

$b)$ Assume that $N\subseteq zd(\Re)$ is not an $r$-hyperideal.  Then there exist $r\in r(\Re)$ and $x\notin N$ such that $r\cdot x\in N$. We have $ann(r)=0.$ Let $J=(N:x)$
and $K=(N:J).$ It follows $r\in J.$ Since $ann(r)=0$, then $J\cap r(\Re
)\neq\emptyset$ and $r\notin N.$ Since $N\subseteq zd(\Re),$ then $ann(x)\neq 0.$
Therefore, $x\in K$. Thus, there exist $b,c\in\Re$ such that $b\cdot x\in N$
and $c\cdot J\subseteq N.$ Hence, $b\cdot x\cdot c\cdot J\subseteq N.$ Then
$x\cdot J\cdot K\subseteq N$ and $J\cdot K\subseteq N.$
\end{proof}

\begin{proposition}
$a)$ Let $\Re$ be a hyperring and $N$ be a hyperideal of $\Re$ with
$N\cap r(\Re)\neq\emptyset.$ If $J$ and $K$ are $r$-hyperideals of $\Re$ such
that $N\cdot J=N\cdot K$ or $N\cap J=N\cap K,$ then $J=K.$

$b)$ Let $\Re$ be a hyperring and $N, M$ be hyperideals of
$\Re$ with $M\cap r(\Re)\neq\emptyset.$ If $N\cdot M$ is $r$-hyperideal of $\Re
$, then $N=N\cdot M.$ In addition, $N$ is an $r$-hyperideal.
\end{proposition}

\begin{proof}
$a)$ Let $J$ and $K$ be $r$-hyperideals of $\Re$ and $N\cdot J=N\cdot K.$ From
Lemma 2, since $N\cdot J=N\cdot K\subseteq K,$ $J\subseteq K.$ Therefore
$K\subseteq J.$ Then, $J=K.$

$b)$ Let $N\cdot M$ be $r$-hyperideal of $\Re$ and $M\cap r(\Re)\neq\emptyset.$
Since $N\cdot M\subseteq M\cdot N,$ from Lemma 2, $N\subseteq N\cdot M.$
Therefore, obviously, $N\cdot M\subseteq N.$ Then $N=N\cdot M.$
\end{proof}

\begin{theorem}
Let $I_{1},I_{2},...,I_{n}$ be prime hyperideals of $\Re,$ which are not
comparable. If $\ \bigcap\limits_{i=1}^{n}I_{i}$ is an $r$-hyperideal, then
$I_{i}$ is an $r$-hyperideal, for $i=1,...,n.$
\end{theorem}

\begin{proof}
Let $r, x\in \Re$ such that $r\cdot x\in I_{i}$ and $ann(r)=0$. Let $y\in(\prod\nolimits_{j\neq
i}I_{j})\backslash I_{i}.$ Then, $r\cdot x\cdot y\in\bigcap\limits_{i=1}%
^{n}I_{i}$ . We get that $x\cdot y\in\bigcap\limits_{i=1}^{n}I_{i}$, since
$\bigcap\limits_{i=1}^{n}I_{i}$ is an $r$-hyperideal and $ann(r)=0.$ Thus,
$x\cdot y\in I_{i}.$ We conclude that $x\in I_{i},$ since $y\notin I_{i}$ and
$I_{i}$'s are prime. Hence, $I_{i}$ is an $r$-hyperideal.
\end{proof}

Let $\rho:\Re\rightarrow S$ be a homomorphism. We investigate whether the
image of an $r$-hyperideal and the inverse image of an $r$-hyperideal are $r$-hyperideal.

\begin{theorem}
Let $\rho:\Re\rightarrow S$ be a good epimorphism such that $Ker$
$\rho\subseteq N$. If $N$ is an $r$-hyperideal of $(\Re,+,\cdot),$ then
$\rho(N)$ is an $r$-hyperideal of the hyperring $(S,\oplus,\odot).$
\end{theorem}

\begin{proof}
Obviously, $\rho(N)$ is hyperideal of $S.$ Let $b_{1}\odot b_{2}\in\rho(N)$
and $ann(b_{1})=0_{S},$ for $b_{1},b_{2}\in S.$ Since $\rho$ \ is onto, then there
exist $a_{1},a_{2}\in\Re$ such that $b_{1}=\rho(a_{1})$ and $b_{2}=\rho
(a_{2})$. Then $b_{1}\odot b_{2}=\rho(a_{1})\odot\rho(a_{2})=\rho(a_{1}\cdot
a_{2})=\rho(x)\in f(N),$ for some $x\in N.$

$0\in$ $\rho(a_{1}\cdot a_{2})\ominus\rho(x)=\rho(a_{1}\cdot a_{2}-x).$ Then
there exists $t\in a_{1}\cdot a_{2}-x$ such that $\rho(t)=0.$ We have
$a_{1}\cdot a_{2}\in t+x\subseteq Ker\rho+N\subseteq N+N\subseteq N.$ Thus,
$a_{1}\cdot a_{2}\in N.$

Let $ann(a_{1})\neq0.$ Then there exists $0\neq c\in\Re$ such that $a_{1}\cdot
c=0_{\Re}.$ So, $\rho(a_{1}\cdot c)=\rho(a_{1})\odot\rho(c)=\rho(0_{\Re}%
)=0_{S}.$ Since $\rho(c)\neq0_{S}$, $ann(\rho(a_{1}))=ann(b_{1})=0_{S}$ and
this is a contradiction$.$ This implies that $ann(a_{1})=0.$ Since $N$ is
an $r$-hyperideal, then $a_{2}\in N.$ Therefore, $b_{2}=\rho(a_{2})\in\rho(N).$
\end{proof}

\begin{theorem}
Let $\rho:\Re\rightarrow S$ be a good monomorphism. If $M$ is an $r$-hyperideal
of $(S,\oplus,\odot),$ then $\rho^{-1}(M)=M^{c}$ is an $r$-hyperideal of
$(\Re,+,\cdot).$
\end{theorem}

\begin{proof}
$\rho^{-1}(M)$ is a hyperideal \cite{X28}. Let $r\cdot x\in\rho^{-1}(M)$ and
$ann(r)=0.$ Then $\rho(r)\odot f(x)=\rho(r\cdot x)\in M.$ If $ann(\rho(r))\neq0,$ then
there exists $0\neq\rho(s)\in M$ such that $\rho(r)\odot\rho(s)=\rho(r\cdot
s)=0_{S}.$ This means that there exists a $0\neq s\in\rho^{-1}(M)$ such that
$r\cdot s=0_{\Re}.$ Then, $ann(r)\neq0$. This is a contradiction, so
$ann(\rho(r))=0.$ Since $M$ is an $r$-hyperideal, then $\rho(x)\in M$ and therefore
$x\in\rho^{-1}(M).$
\end{proof}

\begin{theorem}
Let $\Re$ be a hyperring. The following statements are equivalent:

$a)$ $\Re$ is a hyperdomain.

$b)$ The only $r$-hyperideal of $\Re$ is the zero hyperideal.

$c)$ $ann(x\cdot y)=$ $ann(x)\cup ann(y),$ for each $x,y\in\Re.$
\end{theorem}

\begin{proof}
($a)\Rightarrow(b)$ Let us suppose that $\Re$ is hyperdomain and $(0)\neq N$ is a
proper hyperideal of $\Re.$ Then, there exists $0\neq\Re\in N$. Since $\Re$ is
a hyperdomain, then we have $ann(r)=0$. This contradicts the ideal being proper.
Then the zero ideal is the only $r$-hyperideal.

($b)\Rightarrow(c)$ From Corollary 1(f), $ann(x)$ is an $r$-hyperideal. Assume that
the zero hyperideal is the only $r$-hyperideal of $\Re.$ Hence, $ann(x)=0.$ So
$ann(x\cdot y)=$ $ann(x)\cup ann(y)=0,$ for every $x,y\in\Re.$

($c)\Rightarrow(a)$ Let $x\cdot y=0,$ $x,y\in\Re.$ Thus, $ann(x\cdot
y)=ann(0)=$ $ann(x)\cup ann(y)=\Re$. This means $1_{\Re}\in ann(x)\cup
ann(y).$ Then $x=\{0\}$ or $y=\{0\}.$ Hence, $\Re$ is a hyperdomain.
\end{proof}

\begin{proposition}
Let $\Re$ be a hyperring and $\{e_{i}:i\in\Lambda\}$ be a set of idempotents
of $\Re$. If $N$ is proper and $N=\sum\limits_{i\in\Lambda}e_{i}\cdot\Re$, then
$N$ is an $r$-hyperideal.
\end{proposition}

\begin{proof}
Let $r\cdot x\in N$ and $ann(r)=0.$ We get that $r\cdot x\in\sum
\limits_{k=1}^{n}e_{i_{k}}r_{i_{k}}$, since $N=\sum\limits_{i\in\Lambda}%
e_{i}\cdot\Re,$ for $i_{1},...,i_{n}\in\Lambda$ and $r_{i_{1}},...,r_{i_{k}%
}\in\Re.$ Let $y=\prod\limits_{k=1}^{n}(1-e_{i_{k}}).$ Since $r\cdot x\cdot
y=0$ and $ann(r)=0,$ then $x\cdot y=0.$ Thus there exists $s\in N$ such that
$y\in1-s.$ If $x\cdot y\in x\cdot(1-s)=0,$ then $x=x\cdot s\in N.$
\end{proof}

\begin{proposition}
$a)$ Let $\Re$ be a hyperring and $x,y\in\Re$ with $1_{\Re}\in x+y.$ Then
$N=ann(x)+ann(y)$ is an $r$-hyperideal.

$b)$ Let $\Re$ be a reduced hyperring, $I\in Min(\Re)$ and $e\in\Re$ be an
idempotent element. We infer that $N=I+ann(e)$ is an $r$-hyperideal.
\end{proposition}

\begin{proof}
$a)$ Let $a\cdot b\in N$ and $ann(a)=0.$ Then, there exists $r\in ann(x)$ and
$s\in ann(y)$ such that $a\cdot b\in r+s.$ Obviously, $r\cdot x=0$ and $s\cdot
y=0.$ We get that $a\cdot b\cdot x\cdot y\in(r+s)\cdot(x\cdot y)=r\cdot(x\cdot
y)+s\cdot(x\cdot y)=y\cdot(x\cdot r)+x\cdot(y\cdot s)=b\cdot0+x\cdot0=0$ and
by our assumption, $b\cdot x\cdot y=0$. Thus, $b\cdot x\in ann(y)$ or $b\cdot
y\in ann(x).$ Hence, $b=b\cdot 1_{\Re}\in b\cdot(x+y)=(b\cdot x)+(b\cdot
y)\subseteq ann(x)+ann(y)=N$ and therefore $b\in N.$

$b)$ Let $r\cdot x\in N$ and $ann(r)=0.$ We have $r\cdot x\in a+b$ such that
$a\in I$ and $b\in ann(e).$ Since $I\in Min(\Re),$ then $I$ is an $r$-hyperideal. There
is $y\notin I$ such that $a\cdot y=0.$ $e\cdot y\cdot r\cdot x\in e\cdot
y\cdot(a+b)=(e\cdot y\cdot a)+(e\cdot y\cdot b)=0.$ Since $\Re$ is a reduced
hyperring, then every idempotent element is $0.$ By our assumption, $e\cdot y\cdot
x\in I.$ Since $y\notin I,$ then $e\cdot x\in I.$ We write $x\in x+(e\cdot
x)-(e\cdot x),$ $x\in(e\cdot x)+x-(e\cdot x)=e\cdot x+x\cdot(1_{\Re
}-e)\subseteq I+ann(e)=N.$ Then $N$ is an $r$-hyperideal.
\end{proof}

\begin{corollary}
a) If there is a hyperideal $K$ of $\Re$ such that $ann(N)+ann(M)=ann(K),$ for
$N,M$ hyperideals, then $ann(N)+ann(M)$ is $r$-hyperideal.

b) The direct sum of $r$-hyperideals is $r$-hyperideal. Therefore, if $N=M\oplus
K$ and $N$ is $r$-hyperideal, then $M$ and $K$ are $r$-hyperideals

c) Let $\Re$ be a reduced hyperring. Then $soc(\Re)$ is an $r$-hyperideal.
Therefore, there is a  hyperideal $M$ of $\Re$ such that $soc(\Re)=M^{c}%
.$
\end{corollary}

\begin{proposition}
Let $N$ be a hyperideal of $\Re.$ $N$ is a $pr$-hyperideal if and only if
$\sqrt{N}$ is $r$-hyperideal.
\end{proposition}

\begin{proof}
Similarly with \cite{X26}, let $N$ be $pr$-hyperideal, $x\cdot y\in\sqrt{N}$ and
$ann(x)=0.$ There is $n\in%
\mathbb{N}
$ such that $(x\cdot y)^{n}\in N.$ Hence $x^{n}\cdot y^{n}\in N$ and
$ann(x^{n})=0.$ Since $N$ is $pr$-hyperideal, then $y^{nm}\in N,$ for $m\in%
\mathbb{N}
.$ Thus, $y\in\sqrt{N}.$

Conversely, let $x\cdot y\in N$ and $ann(x)=0.$ Since $N\subseteq\sqrt{N},$ then 
$x\cdot y\in\sqrt{N}.$ Therefore, since $\sqrt{N}$ is $r$-hyperideal, then $b\in
\sqrt{N}.$ Thus, $y^{n}\in N,$ for $n\in%
\mathbb{N}
.$
\end{proof}

\begin{proposition}
Let $\Re$ be a hyperring and $nil(\Re)$ be a set of nilpotents of $\Re$. Then
$nil(\Re)$ is an $r$-hyperideal.
\end{proposition}

\begin{proof}
$nil(\Re)=\{a\in\Re:0=a^{n},n\in%
\mathbb{N}
\}=\sqrt{0}.$ Obviously,$\sqrt{0}$ is $r$-hyperideal. Then $nil(\Re)$ is also $r$-hyperideal.
\end{proof}

From Theorem 1 and Proposition 4, we obtain the next corollary.

\begin{corollary}
Let $N$ be a hyperideal of $\Re.$ The following statements are equivalent:

a) $N$ is $pr$-hyperideal.

b)\ $(r\cdot\Re)\cap\sqrt{N}=r\cdot\sqrt{N},$ for any $r\in\Re(\Re).$

c) $\sqrt{N}=\sqrt{(N:r)},$ for any $r\in r(\Re)$ and $r\notin N.$

d) $N=M^{c},$ $M$ is a primary hyperideal of $Q(\Re).$
\end{corollary}

\begin{theorem}
Every $z^{0}$-hyperideal is an $r$-hyperideal.
\end{theorem}

\begin{proof}
Similar to \cite{X26}, let $x\cdot y\in N$ and $ann(x)=0.$ Then, $ann(x\cdot
y)=ann(y).$ Since $N$ is $z^{0}$-hyperideal, then $y\in N$ and $N$ is an $r$-hyperideal.
\end{proof}

\begin{theorem}
Every hyperideal which contains zero divisors is contained by a prime $r$-hyperideal.
\end{theorem}

\begin{proof}
It is trivial.
\end{proof}

\begin{theorem}
If $N$ is a hyperideal and $I\in Min(N)$ in hyperring $\Re,$ then $I$ is an $r$-hyperideal.
\end{theorem}

\begin{proof}
Let $x\cdot y\in I$ and $ann(x)=0.$ There exists $a\notin I$ and $n\in%
\mathbb{N}
$ such that $a\cdot(x\cdot y)^{n}=a\cdot x^{n}\cdot y^{n}\in N.$ Since $N$ is
$r$-hyperideal and $ann(x^{n})=0,$ then $a\cdot y^{n}\in N\subseteq I.$ Since
$a\notin I,$ then $y^{n}\in I.$ Thus $y\in I.$
\end{proof}

\begin{proposition}
Let $N$ be an $r$-hyperideal and $N\subseteq M.$ If $M/N$ is $r$-hyperideal of
$\Re/N,$ then $M$ is $r$-hyperideal of $\Re.$
\end{proposition}

\begin{proof}
Let $M/N$ be $r$-hyperideal of $\Re/N,$ $(a+N)\cdot(b+N)\subseteq M/N$ and
$ann(a+N)=0$ for $a+N,b+N\in\Re/N.$

$(a+N)\cdot(b+N)\subseteq a\cdot(b+N)+N\cdot(b+N)\subseteq(a\cdot b)+(a\cdot
N)+(N\cdot b)+N\subseteq(a\cdot b)+N\subseteq M+N.$ So, $a\cdot b\in M.$ Since
$ann(a)=0$ and $b\in M$, $M$ is an $r$-hyperideal.
\end{proof}

In the following, we investigate the maximal hyperideal and $r$-hyperideal relation.

\begin{proposition}
Every maximal $r$-hyperideal is a prime hyperideal.
\end{proposition}

\begin{proof}
Obvious from \cite{X26}.
\end{proof}

\begin{proposition}
If every prime hyperideal of $\Re$ is an $r$-hyperideal, every maximal
hyperideal of $\Re$ is an $r$-hyperideal.
\end{proposition}

\begin{proof}
Suppose that $M$ is a maximal hyperideal of $\Re$. We know that every maximal
hyperideal is a prime hyperideal. Let $a\cdot b\in M$ with $ann(a)=0.$ Then,
$b\in(M:a).$ Since $M$ is maximal, then $M=(M:a)$ and $b\in M.$
\end{proof}

\begin{proposition}
Let $\Re$ be a reduced hyperring satisfying the Property A. Then every maximal
$r$-hyperideal is an $z^{0}$-hyperideal.
\end{proposition}

\begin{proof}
Let $I$ be a maximal $r$-hyperideal, $I\subseteq zd(\Re).$ There exists a
$z^{0}$-hyperideal $J$ such that $I\subseteq J.$ Then $J$ is an $r$-hyperideal.
Since $I$ is maximal, then $I=J.$ Thus, $I$ is an $z^{0}$-hyperideal.
\end{proof}

\section{Generalizations of $r$-Hyperideals}

In this section, we introduce a generalization of $r$-hyperideals in commutative Krasner
hyperrings. $Id(\Re)$ denotes the hyperideals of $\Re.$ Initially, we give the
definition of $\phi-$prime and $\phi-$primary hyperideal. Similarly, we give
the definitions of $\phi-r$-hyperideal, $\phi-pr$-hyperideal, $\phi-$pure
hyperideal and $\phi-$von Neumann regular hyperideal in $\Re.$ Afterwards, we
investigate some properties of $\phi-r$-hyperideal of $\Re$.

\begin{definition}
Let $\Re$ be a Krasner hyperring$,$ $\phi:Id(\Re)\rightarrow Id(\Re
)\cup\{\emptyset\}$ be a function and $\emptyset\neq N\in Id(\Re).$ Then $N$
is said to be a $\phi-$prime (resp. $\phi-$primary) hyperideal of $\Re$ if
whenever $r,s\in\Re$ and $\Re\cdot s\in N-$ $\phi(N),$ then $r\in N$ or $s\in
N$ (resp. $r\in N$ or $s^{n}\in N$).
\end{definition}

\begin{definition}
Let $\Re$ be a Krasner hyperring, $N$ be a proper hyperideal of $\Re$ and
$\phi:Id(\Re)\rightarrow Id(\Re)\cup\{\emptyset\}$ be a function. $N$ is
called $\phi-r$-hyperideal (resp., $\phi-pr$-hyperideal) if $a\cdot b\in
N-\phi(N)$ with $ann(a)=0$ implies that $b\in N$ (resp., $b^{n}\in N$), for
$a,b\in\Re.$
\end{definition}

Let $J$ be a $\phi-r$-hyperideal of $\Re.$ Then:

i) If $\phi(J)=\varnothing,$ for all $J\in Id(\Re),$ then $J$ is called
$\phi_{\emptyset}-r$-hyperideal of $\Re$ ($r$-hyperideal).

ii) If $\phi(J)=0,$ for all $J\in Id(\Re),$ then $J$ is called $\phi_{0}%
-r$-hyperideal of $\Re$ (weakly $r$-hyperideal).

iii) If $\phi(J)=J,$ for all $J\in Id(\Re),$ then $J$ is called $\phi_{1}%
-r$-hyperideal of $\Re$ (any hyperideal).

iv) If $\phi(J)=J^{n},$ for all $J\in Id(\Re)$ and $n\geq2,$ then $N$ is called
$\phi_{n}-r$-hyperideal ($n$-almost $r$-hyperideal). If $n=2,$ then $J$ is called
almost $r$-hyperideal of $\Re$.

v) \ If $\phi(J)=\cap_{i=1}^{\infty}J^{i},$ for all $J\in Id(\Re),$ then $J$ is
called $\phi_{\omega}-r$-hyperideal ($\omega-r$-hyperideal) of $\Re.$

\begin{definition}
Let $\Re$ be a Krasner hyperring, $N$ be a proper ideal of $\Re$ and
$\phi:Id(\Re)\rightarrow Id(\Re)\cup\{0\}$ be a function.

(1) $N$ is called $\phi-$pure hyperideal if there exists $b\in N$ such that
$a=a\cdot b,$ for every $a\in N-\phi(N).$

(2) $N$ is called $\phi-$von Neumann regular hyperideal if there exists $b\in
N$ such that $a=a^{2}\cdot b,$ for every $a\in N-\phi(N).$
\end{definition}

\begin{theorem}
Let $N\in Id(\Re).$ Then the following statements hold:

(i) Let $\psi_{1}$ and $\psi_{2}$ be two functions. $\psi_{1,2}%
:Id(\Re)\longrightarrow Id(\Re)\cup\{\emptyset\}$ such that $\psi_{1}\leq
\psi_{2}.$ If $N$ is a $\psi_{1}-r$-hyperideal, then $N$ is a $\psi_{2}-r$-hyperideal.

(ii) Every $\phi-r$-hyperideal is a $\phi-pr$-hyperideal.

(iii) $N$ is an $r$-hyperideal $\Rightarrow N$ is a weakly
$r$-hyperideal $\Rightarrow$ $N$ is an $\omega-r$-hyperideal $\Rightarrow$ $N$ is
an $(n+1)$-almost $r$-hyperideal $\Rightarrow N$ is an $n$-almost $r$-hyperideal 
$\Rightarrow N$ is an almost $r$-hyperideal.

(iv) Every $\phi-$pure hyperideal is a $\phi-r$-hyperideal.

(v) Every $\phi-$von Neumann regular hyperideal is a $\phi-r$-hyperideal.
\end{theorem}

\begin{proof}
Straighforward.
\end{proof}

\begin{theorem}
Let $\Re$ be a Krasner hyperring, $N$ be a proper ideal of $\Re.$ Then the
following statements are equivalent:

i) $N$ is $\phi-r$-hyperideal;

ii) $(N:r)=N\cup(\phi(N):r),$ $r\in r(\Re)\backslash N;$

iii) $(N:r)=N$ or $(N:r)=(\phi(N):r),$ $r\in r(\Re)\backslash N;$
\end{theorem}

\begin{proof}
($i)\Rightarrow(ii)$ We know that $N\subseteq(N:r)$ and $(\phi(N):r)\subseteq
(N:r).$ Let $s\in(N:r).$ Then $r\cdot s\in N.$ If $r\cdot s\in$ $\phi(N),$ then
$s\in(\phi(N):r).$ On the other hand, if $r\cdot s\notin$ $\phi(N),$ since $N$
is $\phi-r$-hyperideal, then $s\in N.$

($ii)\Rightarrow(iii)$ It is evident.

($iii)\Rightarrow(i)$ Let $r,s\in\Re$ such that $r\cdot s\in N-$ $\phi(N)$ and
$r$ be regular. Then, $(N:r)=N$ or $(N:r)=(\phi(N):r)$.

Case 1: Assume that $(N:r)=N.$ Since $r\cdot s\in N,$ then $s\in$ $(N:r)=N.$

Case 2: Assume that $(N:r)=(\phi(N):r)$. Since $r\cdot s\in N,$ then $s\in(N:r).$
Thus $r\cdot s\in$ $\phi(N).$ This is a contradiction.
\end{proof}

\begin{definition}
Let $N$ be a proper hyperideal of $\Re$, $J$ and $K$ be hyperideals of $\Re$
such that $J\cdot K\subseteq N,$ $J\cdot K\nsubseteq\phi(N)$ and $ann(J)=0$
implies that $K\subseteq N$. Then $N$ is called a strongly $\phi-r$-hyperideal.
\end{definition}

\begin{theorem}
Every strongly $\phi-r$-hyperideal is an $\phi-r$-hyperideal.
\end{theorem}

\begin{proof}
Similar to \cite{X3}, let $a,b\in\Re$ such that $a\cdot b\in N-$ $\phi(N)$ and
$ann(a)=0.$ Then $<a>\cdot<b>\subseteq N-\phi(N).$ Obviously, $ann(<a>)=0.$
Since $N$ is strongly $\phi-r$-hyperideal, then $<b>\subseteq N.$ Therefore, $b\in N.$
\end{proof}

\begin{theorem}
Let $\Re$ satisfy the s.a.c., $N$ be a proper hyperideal of $\Re$ and
$\phi(N)$ be an $r$-hyperideal of $\Re.$ $N$ is $\phi-r$-hyperideal if and only
if for every hyperideal $J$ and $K$ of $\Re$ such that $J$ is finitely
generated, whenever $J\cdot K\subseteq N$, $J\cdot K\nsubseteq\phi(N)$ and
$ann(J)=0,$ implies that $K\subseteq N.$

\begin{proof}
$\Rightarrow:$ Let $N$ be a $\phi-r$-hyperideal, $\phi(N)$ be an $r$-hyperideal,
$J$ be finitely generated, $J$ and $K$ be hyperideals of $\Re$ such that
$J\cdot K\subseteq N$ and $J\cdot K\nsubseteq\phi(N).$ Assume that
$K\nsubseteq N.$ Then there exists an element $0\neq b\in K-N.$ Since
$\Re$ satisfy the s.a.c and $J$ is finitely generated, then there is an element
$a\in J$ with $ann(J)=ann(a)=0.$ Let $a\cdot b\in N.$ If $a\cdot b\notin
\phi(N)$, since $N$ is $\phi-r$-hyperideal, then $b\in N$ which is a contradiction.
If $a\cdot b\in\phi(N),$ since $\phi(N)$ is an $r$-hyperideal, then $b\in\phi(N)$
which is a contradiction. Thus, $K\subseteq N.$

$\Leftarrow:$ It is obvious.
\end{proof}
\end{theorem}

Now, we investigate the relation between the hyperideal $N$ and the set of all
zero divisors, while $\phi(N)$ is an $r$-hyperideal of $\Re.$

\begin{theorem}
Let $\phi(N)$ be an $r$-hyperideal of $\Re$ and $N$ be a proper hyperideal of
$\Re.$ If $N$ is $\phi-r$-hyperideal, then $N\subseteq zd(\Re)$.
\end{theorem}

\begin{proof}
Let $N\nsubseteq zd(\Re).$ Then there exists an element $a\in N$ $-zd(\Re).$
So, $ann(a)=0.$ We have $a=a\cdot1_{\Re}\in N.$ If $a\notin\phi(N),$ then
$1_{\Re}\in N,$ which is a contradiction. If $a\in\phi(N),$ then $1_{\Re}\in
\phi(N)\subseteq N,$ since $\phi(N)$ is an $r$-hyperideal.
\end{proof}

\begin{theorem}
Let $\phi(N)$ be an $r$-hyperideal of $\Re$ and $N$ be a prime hyperideal of
$\Re.$ $N$ is $\phi-r$-hyperideal if and only if $N\subseteq zd(\Re)$.

\begin{proof}
Let $\phi(N)$ be an $r$-hyperideal of $\Re$ and $N$ be a prime hyperideal.

$\Rightarrow:$ It is obvious from previous theorem.

$\Leftarrow:$ Let $a\cdot b\in N-$ $\phi(N)$ and $ann(a)=0.$ Since $N$ is a
prime hyperideal, $a\in N$ or $b\in N.$ Since $N\subseteq zd(\Re)$, then $a\notin
N.$ Therefore $b\in N.$
\end{proof}
\end{theorem}

In a hyperideal $N,$ nilpotent elements are also zero divisors. So we can
change the theorem as if $N$ is a hyperideal of which elements are nilpotent, then
$N$ is $\phi-r$-hyperideal if and only if $N$ is $\phi$-primary hyperideal.

\begin{theorem}
If $N$ is $\phi-r$-hyperideal, then $N-\phi(N)\subseteq zd(\Re).$
\end{theorem}

\begin{proof}
Let us suppose that $N-\phi(N)\nsubseteq zd(\Re).$ Then there exists an element $a$ in
$N-\phi(N)$ such that $ann(a)=0.$ Since $a=a\cdot1\in N-\phi(N)$ and $N$ is
$\phi-r$-hyperideal, then $1\in N.$ That is a contradiction.
\end{proof}

\begin{theorem}
Let $\phi(N)$ be an $r$-hyperideal. $N$ is $\phi-r$-hyperideal if and only if
$N$ is an $r$-hyperideal.
\end{theorem}

\begin{proof}
($\Rightarrow:)$ Let $a\cdot b\in N$ and $ann(a)=0.$ If $a\cdot b\in\phi(N),$
since $\phi(N)$ is an $r$-hyperideal, then $b\in\phi(N)\subseteq N.$ If $a\cdot
b\notin\phi(N),$ since $N$ is $\phi-$r-hyperideal, then $b\in N.$

($\Leftarrow:)$ It is obvious.
\end{proof}

Similar to \cite{X3}, if $N$ is a hyperideal of $\Re,$ we define $\phi
_{N}:Id(\Re/N)\longrightarrow Id(\Re/N)\cup\{\emptyset\}$ by $\phi_{N}(M/N)=$
$(\phi(M)+N)/N$ for each hyperideal $N\subseteq M$. If $\phi(N)=\emptyset,$
$\phi_{N}(M/N)=\emptyset$. We see that $\phi_{N}(M/N)\subseteq M/N.$

\begin{theorem}
Suppose that $N\subseteq r(\Re)$ and $M$ are hyperideals in $\Re$ such that
$N\subseteq M.$ If $M$ is $\phi-r$-hyperideal of $\Re$, then $M/N$ is
$\phi_{N}-r$-hyperideal of $\Re/N.$
\end{theorem}

\begin{proof}
Suppose that $(r+N)\cdot(s+N)\subseteq$ $M/N-\phi_{M}(M/N)=M/N-(\phi(M)+N)/N$
and $ann_{\Re/N}(r+N)=0.$ Then $r\cdot s\in M-\phi(M)$ and as $N\subseteq
r(\Re),$ $ann_{\Re}(r)=0$. Since $M$ is $\phi-r$-hyperideal, then $s\in M.$
Therefore, $s+N\subseteq$ $M/N.$
\end{proof}

\begin{theorem}
Suppose that $N$ is an $r$-hyperideal and $M$ is a hyperideal in $\Re$ such that
$N\subseteq M.$ If $M/N$ is $r$-hyperideal of $\Re/N$, then $M$ is $\phi
-r$-hyperideal of $\Re.$
\end{theorem}

\begin{proof}
Suppose that $r\cdot s\in$ $M-\phi(M)$ with $ann(r)=0.$ If $r\cdot s\in
N$, since $N$ is an $r$-hyperideal, then $s\in N\subseteq M.$ If $r\cdot s\in M-N,$ then
$r\cdot s+N=(r+N)\cdot(s+N)\subseteq M/N.$ In addition, since $ann(r)=0,$ then
$ann(r+N)=0.$ Therefore, since $M/N$ is $r$-hyperideal of $\Re/N,$ then $s+N\subseteq
M/N.$ Hence, $s\in M,$ as desired.
\end{proof}

Let $\phi:Id(\Re)\rightarrow Id(\Re)\cup\{0\}$ be a function. If $N$ and $M$
are hyperideals of $\Re$ such that $N\subseteq M$ implies that $\phi
(N)\subseteq\phi(M)$, then $\phi$ is said to preserve the order.

\begin{theorem}
Let $A$ be a subset of $\Re,$ $A\neq\emptyset$ and $\phi$ preserves the order. If
$N$ is $\phi-r$-hyperideal such that $A\nsubseteq N,$ then $(N:A)$ is $\phi-r$-hyperideal.
\end{theorem}

\begin{proof}
Suppose that $a\cdot b\in(N:A)-\phi(N:A)$ and $ann(a)=0.$ Since $\phi$
preserves order, then $\phi(N)\subseteq\phi(N:A).$ We have $a\cdot b\cdot A\subseteq
N-\phi(N)$. Therefore, $b\cdot A\subseteq N$ and $b\in(N:A).$
\end{proof}

\begin{corollary}
Let $\{N_{i}\}_{i\in\Lambda}$ directed collection of ascending chain $\phi
-r$-hyperideals of $\Re$ and $\phi$ preserves the order. Then $N=\underset
{N\in\Lambda}{\cup}N_{i}$ is $\phi-r$-hyperideal.
\end{corollary}

\begin{proof}
Let $a\cdot b\in N-\phi(N)$ and $ann(a)=0.$ Then $a\cdot b\in N_{k}$ and
$a\cdot b\notin\phi(N_{N}),$ for $k\in\Lambda.$ Since $N_{k}$ is $\phi
-r$-hyperideal, then $b\in N_{k}\subseteq N.$\\
\end{proof}

{\textbf{Compliance with Ethical Standards}}\\

\textbf{Conflict of Interest} All authors declare that they
have no conflict of interest.\\

\textbf{Ethical approval} This article does not contain any
studies with human participants or animals performed
by any of the authors.\\

\textbf{Data Availability Statement}

Data sharing not applicable to this article as no data-sets were generated or analysed during the current study.


\begin{thebibliography}{99}                                                                                               %


\bibitem {X1} {R. Ameri, T. Nozari}. A new characterization of fundamental relation
on hyperrings. Int. J. Contemp. Math. Sci.
Vol. 5 (2010), no. 13-16, 721-738.

\bibitem {X2} {D.D. Anderson, M. Bataineh}. Generalizations of Prime Ideals.
Comm. Algebra 36 (2008), 686-696.

\bibitem {X3} {E.A. Ugurlu}. Generalizations of $r$-ideals of
commutative rings. https://arxiv.org/abs/2006.12261.

\bibitem {X4} {A. Asokkumar}. Derivations in hyperrings and prime hyperrings.
Iran. J. Math. Sci. Inform. Vol. 8 (2013), no. 1, 1-13.

\bibitem {X5} {A. Asokkumar, M. Velrajan}. Characterizations of regular hyperrings.
Ital. J. Pure Appl. Math. 22 (2007), 115-124.

\bibitem {X6} {A. Asokkumar, M. Velrajan}. Hyperring of matrices over a regular
hyperring. Ital. J. Pure Appl. Math. 23 (2008), 113-120.

\bibitem {X7} {A. Asokkumar, M. Velrajan}. Von Neumann Regularity on Krasner
Hyperring. Algebra, Graph Theory and their Applications, Editors: T. Tamizh
Chelvam, S. Somasundaram and R. Kala, Narosa Publishing House, New Delhi,
India 2010; 9-19.

\bibitem {X8} {A.R. Barghi}. A class of hyperrings. J. Discrete Math. Sci. Cryptogr. Vol. 6 (2003), no. 2-3, 227-233.

\bibitem {X9} {P. Corsini}. Prolegomena of hypergroup theory. Supplement to Riv.
Mat. Pura Appl. {Aviani Editore, Tricesimo} 1993; 215 pp. ISBN: 88-7772-025-5.

\bibitem {X10} {A.Y. Darani}. Generalizations of primary ideals in commutative
rings. Novi Sad J. Math. 42 (2012), 27-35.

\bibitem {X11} {U. Dasgupta}. On prime and primary hyperideals of a multiplicative
hyperring. An. Ştiinţ. Univ. Al. I. Cuza Iaşi. Mat. (N.S.) 58 (2012), no. 1, 19-36.

\bibitem {X12} {B. Davvaz}. Isomorphism theorems of hyperrings. Indian J. Pure Appl. Math. Vol. 35 (2004), no. 3, 321-331.

\bibitem {X13} {B. Davvaz, V. Leoreanu-Fotea}. Hyperring theory and applications.
International Academic Press, Palm Harbor, Fla, USA 2007.

\bibitem {X14} {B. Davvaz, A. Salasi}. A realization of hyperrings. Comm. Algebra 34 (2006), no. 12, 4389-4400.

\bibitem {X15} {M. De Salvo}. Hyperrings and hyperfields. Ann. Sci. Univ. Clermont-Ferrand II Math. No. 22 (1984), 89-107.

\bibitem {X16} {K. Hila, K. Naka}. On pure hyperradical in semihypergroups. Int. J. Math. Math. Sci. 2012; Article ID 876919, 7 pages doi:10.1155/2012/876919.
\bibitem {Hila} K. Hila, K. Naka, B. Davvaz, On $(k,n)$-absorbing hyperideals in Krasner $(m,n)$-hyperrings. {Q. J. Math.} {69} (2018), no. 3,  1035-1046.

\bibitem {X17} {J.A. Huckaba.} Commutative rings with zero divisors. Monographs and Textbooks in Pure and Applied Mathematics, 117. Marcel Dekker, Inc., New York, 1988. x+216 pp. ISBN: 0-8247-7844-8.

\bibitem {X18} {A. Jaber}. Properties of $\phi-\delta-$primary and 2-absorbing
$\delta-$primary ideals of commutative rings. Asian-Eur. J. Math. Vol. 13 (2020), no. 1: 2050026 (11 pages).

\bibitem {X19} {J. Jun}. Algebraic geometry over hyperrings. arXiv preprint
arXiv:1512.04837, 2015.

\bibitem {X20} {M. Krasner}. A class of hyperrings and hyperfields. Int. J. Math. Math. Sci. Vol. 6 (1983), no. 2, 307-311.
\bibitem{kemprasit} Y. Kemprasit. Multiplicative interval semigroups on $R$ admitting hyperring structure. Ital. J. Pure Appl. Math. 11 (2002), 205-21.
\bibitem {X21} {T.G. Lucas}. Two annihilator conditions: property (A) and (a.c.).
Comm. Algebra 14 (1986), no. 3, 557-580.

\bibitem {X22} {F. Marty}. Sur une generalization de la notion de group.
Proceedings of the 8th Congres Math. Scandinaves, Stockholm, Sweden 1934; 45-49.

\bibitem {X23} {S. Mirvakili, B. Davvaz}. Relations on Krasner (m,n)-hyperrings.
European J. Combin. 31 (2010), 790-802.

\bibitem {X24} {J. Mittas}. Hyperanneaux et certaines de leurs proprietes. C. R. Acad. Sci. Paris Sér. A-B 269 (1969), A623–A626.

\bibitem {X25} {J. Mittas}. Hypergroupes canoniques. Math. Balkanica 2 (1972), 165-179.

\bibitem {X26} {R. Mohamadian}. $r$-ideals in commutative rings. Turkish J. Math.  39 (2015), 733-749.

\bibitem {X27} {A. Nakassis}. Recent results in hyperring and hyperfield theory.
Int. J. Math. Math. Sci. 11 (1988), 209-220.

\bibitem {X28} {S. Omidi, B. Davvaz, J. Zhan.} Some properties of $n$-hyperideals in
commutative hyperrings. Journal of Algebraic Hyperstructures and Logical
Algebras Vol. 1 (2020), no. 2, 23-30.

\bibitem {X29} R. Rota. Strongly distributive multiplicative hyperrings. J.
Geom. Vol. 39 (1990), no. 1-2, 130-138.

\bibitem {X30} {D. Stratigopoulos}. Certaines classes d'hypercorps et d'hyperanneaux. Hypergroups, other multivalued structures and their applications (Udine, 1985), 105–110, Univ. Studi Udine, Udine, 1985.

\bibitem {X31} {T. Vougiouklis}. Hyperstructures and their representations.
Hadronic Press Monographs in Mathematics. Hadronic Press, Inc., Palm Harbor, FL, 1994. vi+180 pp. ISBN: 0-911767-76-2.

\end{thebibliography}
\end{document}